\documentclass[psamsfonts]{amsart}
\usepackage{amssymb,amsfonts}
\usepackage[all,arc]{xy}
\usepackage{enumerate}
\usepackage{color}

\usepackage{mathrsfs}

\usepackage{picture}
\usepackage{graphicx}

\usepackage{lscape}
\usepackage{geometry}

\usepackage{pstricks,color}
\usepackage{pst-plot}

\newrgbcolor{tauzerocolor}{0.5 0.5 0.5}

\newrgbcolor{twoextncolor}{0.7 0.7 0}
\newrgbcolor{etaextncolor}{0.5 0 1}
\newrgbcolor{nuextncolor}{0.6 0.3 0}

\newgray{gridline}{0.8}

\newpsobject{twoextn}{psline}{linecolor=twoextncolor}
\newpsobject{etaextn}{psline}{linecolor=etaextncolor}
\newpsobject{nuextncurve}{psbezier}{linecolor=nuextncolor}

\newtheorem{thm}{Theorem}[section]
\newtheorem{cor}[thm]{Corollary}
\newtheorem{prop}[thm]{Proposition}
\newtheorem{lem}[thm]{Lemma}

\theoremstyle{definition}
\newtheorem{defn}[thm]{Definition}

\newtheorem{rem}[thm]{Remark}

\newtheorem{notation}[thm]{Notation}

\makeatletter
\let\c@equation\c@thm
\makeatother
\numberwithin{equation}{section}

\makeatletter
\ifx\SK@label\undefined\let\SK@label\label\fi
 \let\your@thm\@thm
 \def\@thm#1#2#3{\gdef\currthmtype{#3}\your@thm{#1}{#2}{#3}}
 \def\mylabel#1{{\let\your@currentlabel\@currentlabel\def\@currentlabel
  {\currthmtype~\your@currentlabel}
 \SK@label{#1@}}\label{#1}}
 
\makeatother

\title{Some extensions in the Adams spectral sequence and the 51-stem}

\author{Guozhen Wang}
\address{Shanghai Center for Mathematical Sciences, Fudan University, Shanghai, China, 200433}
\email{wangguozhen@fudan.edu.cn}

\author{Zhouli Xu}
\address{Department of Mathematics, Massachusetts Institute of Technology, Cambridge, MA 02139}
\email{xuzhouli@math.mit.edu}

\begin{document}

\maketitle

\begin{abstract}
We show a few nontrivial extensions in the classical Adams spectral sequence. In particular, we compute that the 2-primary part of $\pi_{51}$ is $\mathbb{Z}/8\oplus\mathbb{Z}/8\oplus\mathbb{Z}/2$. This was the last unsolved 2-extension problem left by the recent works of Isaksen and the authors (\cite{Isa1}, \cite{IX}, \cite{WX1}) through the 61-stem.

The proof of this result uses the $RP^\infty$ technique, which was introduced by the authors in \cite{WX1} to prove $\pi_{61}=0$. This paper advertises this technique through examples that have simpler proofs than in \cite{WX1}.
\end{abstract}

\tableofcontents

\section{Introduction}

The computation of the stable homotopy groups of spheres is both a fundamental and a difficult problem in homotopy theory. Recently, using Massey products and Toda brackets, Isaksen \cite{Isa1} extended the 2-primary Adams spectral sequence computations to the 59-stem, with a few $2, \eta, \nu$-extensions unsettled.

Based on the algebraic Kahn-Priddy theorem \cite{KP, Lin}, the authors \cite{WX1} compute a few differentials in the Adams spectral sequence, and proved that $\pi_{61}=0$. The 61-stem result has the geometric consequence that the 61-sphere has a unique smooth structure, and it is the last odd dimensional case. In the article \cite{WX1}, it took the authors more than 40 pages to introduce the method and prove one Adams differential $d_3(D_3) = B_3$. Here $B_3$ and $D_3$ are certain elements in the 60 and 61-stem. Our notation will be consistent with \cite{Isa1} and \cite{WX1}.

In this paper, we show that our technique can also be used to solve extension problems in the Adams spectral sequence. We establish a nontrivial 2-extension in the 51-stem, together with a few other extensions left unsolved by Isaksen \cite{Isa1}. As a result, we have the following proposition.

\begin{prop}
There is a nontrivial 2-extension from $h_0h_3g_2$ to $gn$ in the 51-stem.
\end{prop}

We'd like to point out that this is also a nontrivial 2-extension in the Adams-Novikov spectral sequence.

Combining with Theorem 1.1 of \cite{IX}, which describes the group structure of $\pi_{51}$ up to this 2-extension, we have the following corollary.

\begin{cor}
The 2-primary $\pi_{51}$ is $\mathbb{Z}/8\oplus\mathbb{Z}/8\oplus\mathbb{Z}/2$, generated by elements that are detected by $h_3g_2, P^6h_2$ and $h_2B_2$.
\end{cor}

Using a Toda bracket argument, Proposition 1.1 is deduced from the following $\sigma$-extension in the 46-stem.

\begin{prop}~
\begin{enumerate}
\item
There is a nontrivial $\sigma$-extension from $h_3d_1$ to $N$ in the 46-stem.
\item
There is a nontrivial $\eta$-extension from $h_1g_2$ to $N$ in the 46-stem.
\end{enumerate}
\end{prop}

As a corollary, we prove a few more extensions.

\begin{cor}~
\begin{enumerate}
\item
There is a nontrivial $\eta$-extension from $C$ to $gn$ in the 51-stem.
\item
There is a nontrivial $\nu$-extension from $h_2h_5d_0$ to $gn$ in the 51-stem.
\item
There is a nontrivial $\sigma$-extension from $h_0^2g_2$ to $gn$ in the 51-stem. In particular, the element $gn$ detects $\sigma^3\theta_4$.
\end{enumerate}
\end{cor}

\begin{rem}
In \cite{Isa1}, Isaksen had an argument that implies the nonexistence of the two $\eta$-extensions on $h_1g_2$ and $C$, which is contrary to our results in Proposition 1.3 and Corollary 1.4. Isaksen's argument fails because of neglect of the indeterminacy of a certain Massey product in a subtle way. For more details, see Remark 2.3.
\end{rem}

The proof of the $\sigma$-extension in Proposition 1.3 is the major part of this article: we prove it by the $RP^\infty$ technique as a demonstration of the effectiveness of our method.

The rest of this paper is organized as the following.

In Section 2, we deduce Proposition 1.1 and Corollary 1.4 from Proposition 1.3. We also show the two statements in Proposition 1.3 are equivalent. In Section 3, we recall a few notations from \cite{WX1}. We also give a brief review of how to use the $RP^\infty$ technique to prove differentials and to solve extension problems. In Section 4, we present the proof of Proposition 1.3. In Section 5, we prove a lemma which is used in Section 4. The lemma gives a general connection that relates Toda brackets and extension problems in 2 cell complexes. In the Appendix, we use cell diagrams as intuition for statements of the lemmas in Section 5.

\subsection*{Acknowledgement} We thank the anonymous referee for various helpful comments. This material is based upon work supported by the National Science Foundation under Grant No. DMS-1810638.

\section{the 51-stem and some extensions}

We first establish the following lemma.

\begin{lem}
In the Adams $E_2$-page, we have the following Massey products in the 46-stem:
$$gn = \langle N, h_1, h_2\rangle = \langle N, h_2, h_1\rangle$$
\end{lem}

\begin{proof}
By Bruner's computation \cite{Br2}, there is a relation in bidegree $(t-s,s)=(81,15)$:
$$gnr = mN.$$
We have $Ext^{15,81+15} = \mathbb{Z}/2\oplus\mathbb{Z}/2$, generated by $gnr$ and $h_1x_{14,42}$. Moreover, the element $gnr$ is not divisible by $h_1$, and neither of the generators is divisible by $h_2$.

By Tangora's computation \cite{Tan1}, we have a Massey product in the Adams $E_2$-page,
$$m = \langle r, h_1, h_2\rangle.$$

Therefore,
$$gn \cdot r = m\cdot N = N \cdot \langle r, h_1, h_2\rangle = \langle N\cdot r, h_1, h_2\rangle = r \cdot \langle N, h_1, h_2\rangle$$
with zero indeterminacy. This implies
$$gn = \langle N, h_1, h_2\rangle.$$

Because of the relation $h_2 \cdot N=0$ in $Ext^{9,49+9} = 0$, we also have $$gn \cdot r = m\cdot N = \langle r, h_1, h_2\rangle\cdot N = r \cdot  \langle h_1, h_2, N\rangle.$$
This implies
$$gn = \langle N, h_2, h_1\rangle.$$
\end{proof}

Based on Proposition 1.3, we prove part $(1)$ of Corollary 1.4.

\begin{proof}
By Proposition 1.3, $N$ detects the homotopy class $\sigma^2\{d_1\}$. Then the Massey product
$$gn =\langle N, h_2, h_1\rangle$$
and Moss's theorem \cite{Mos} imply that $gn$ detects a homotopy class that is contained in the Toda bracket
$$\langle \sigma^2\{d_1\}, \nu, \eta\rangle.$$
The indeterminacy of this Toda bracket is
$$\eta \cdot \pi_{50} + \sigma^2\{d_1\} \cdot \pi_5 = \eta \cdot \pi_{50},$$
since $\pi_5=0$. Shuffling this bracket, we have
$$\langle \sigma^2\{d_1\}, \nu, \eta\rangle \supseteq \sigma\{d_1\} \cdot \langle \sigma, \nu, \eta\rangle =0,$$
since $\langle \sigma, \nu, \eta\rangle \subseteq \pi_{12} =0$.

Therefore, $gn$ detects a homotopy class that lies in the indeterminacy, and hence is divisible by $\eta$.

For filtration reasons, the only possibilities are $C$ and $h_5c_1$. However, Lemma 4.2.51 of \cite{Isa1} states that there is no $\eta$-extension from $h_5c_1$ to $gn$. Therefore, we must have an $\eta$-extension from $C$ to $gn$.
\end{proof}

Based on Proposition 1.3, we prove part $(2)$ of Corollary 1.4.

\begin{proof}
By Proposition 1.3, $N$ detects the homotopy class $\sigma^2\{d_1\}$. Then the Massey product
$$gn =\langle N, h_1, h_2\rangle$$
and Moss's theorem \cite{Mos} imply that $gn$ detects a homotopy class that is contained in the Toda bracket
$$\langle \sigma^2\{d_1\}, \eta, \nu\rangle.$$
The indeterminacy of this Toda bracket is
$$\nu \cdot \pi_{48} + \sigma^2\{d_1\} \cdot \pi_5 = \nu \cdot \pi_{48},$$
since $\pi_5=0$. Shuffling this bracket, we have
$$\langle \sigma^2\{d_1\}, \eta, \nu\rangle \supseteq \sigma \cdot \langle \sigma\{d_1\}, \eta, \nu\rangle = \sigma\{d_1\} \cdot \langle \eta, \nu, \sigma\rangle =0,$$
since $\langle \eta, \nu, \sigma\rangle \subseteq \pi_{12} =0$.

Therefore, $gn$ detects a homotopy class that lies in the indeterminacy, and hence is divisible by $\nu$.

For filtration reasons, the only possibility is $h_2h_5d_0$, which completes the proof.
\end{proof}

Now we prove part $(3)$ of Corollary 1.4, and Proposition 1.1.

\begin{proof}
Lemma 4.2.31 from Isaksen's computation \cite{Isa1} states that the 2-extension from $h_0h_3g_2$ to $gn$ is equivalent to the $\nu$-extension from $h_2h_5d_0$ to $gn$. This proves Proposition 1.1.

It is clear that Proposition 1.1 is equivalent to part $(3)$ of Corollary 1.4, since $\sigma$ is detected by $h_3$, and $\sigma^2\theta_4$ is detected by $h_0^2g_2$. (See \cite{BMT,Isa1} for the second fact.)
\end{proof}

In the following Lemma 2.2, we show that the two statements in Proposition 1.3 are equivalent.

\begin{lem}
There is a $\sigma$-extension from $h_3d_1$ to $N$ if and only if there is an $\eta$-extension from $h_1g_2$ to $N$.
\end{lem}

\begin{proof}
First note that there are relations in $Ext$:
$$h_3d_1 = h_1e_1, h_3e_1 = h_1g_2.$$
By Bruner's differential \cite[Theorem 4.1]{Br1}
$$d_3(e_1) = h_1t = h_2^2n,$$
we have Massey products in the Adams $E_4$-page
$$h_3d_1 = h_1e_1 = \langle h_2n, h_2, h_1\rangle, \ h_1g_2 = h_3e_1 = \langle h_3, h_2n, h_2\rangle.$$
Then Moss's theorem implies that they converge to Toda brackets
$$\langle \nu\{n\}, \nu, \eta\rangle, \ \langle \sigma, \nu\{n\}, \nu \rangle.$$
Therefore, the lemma follows from the shuffling
$$\sigma \cdot \langle \nu\{n\}, \nu, \eta\rangle = \langle \sigma, \nu\{n\}, \nu \rangle \cdot \eta.$$
\end{proof}

We give a remark on the two $\eta$-extensions we proved.

\begin{rem}
In Lemma 4.2.47 and Lemma 4.2.52 of \cite{Isa1}, Isaksen showed that there are no $\eta$-extensions from $h_1g_2$ to $N$ or from $C$ to $gn$. Both arguments are based the statement of Lemma 3.3.45 of \cite{Isa1}, whose proof implicitly studied the following motivic Massey product
$$\langle h_1^2, Ph_1h_5c_0, c_0\rangle \ni Ph_1^3h_5e_0$$
in the 59-stem of the motivic Adams $E_3$-page, which therefore converges to a motivic Toda bracket. However, in the motivic Adams $E_3$-page, the element $Ph_1^3h_5e_0$ is in the indeterminacy of this motivic Massey product, since $Ph_1h_5e_0$ is present in the motivic $E_3$-page (it supports a $d_3$ differential). Therefore, we have
$$\langle h_1^2, Ph_1h_5c_0, c_0\rangle = \{Ph_1^3h_5e_0, 0\}$$
instead. The statement of Moss's theorem gives us the convergence of \emph{only one} permanent cycle in the motivic Massey product, therefore, in this case, it is inconclusive.
\end{rem}

\section{The method and notations}

In this section, we recall a few notations from \cite{WX1} and set up terminology that will be used in Section 4.

\begin{notation}
All spectra are localized at the prime 2. Suppose $Z$ is a spectrum. Let $Ext(Z)$ denote the Adams $E_2$-page $Ext_{A}(\mathbb{F}_2, H_*(Z; \mathbb{F}_2))$ that converges to the 2-primary homotopy groups of $Z$. Here $A$ is the mod 2 dual Steenrod algebra.

We now introduce some spectral sequence terminologies. A permanent cycle is a class that does not support any nontrivial differential. A surviving cycle is a permanent cycle that is also not the target of any differential.  


For spectra, let $S^0$ be the sphere spectrum, and $P_1^\infty$ be the suspension spectrum of $RP^\infty$. In general, we use $P_n^{n+k}$ to denote the suspension spectrum of $RP^{n+k}/RP^{n-1}$.

Let $\alpha$ be a class in the stable homotopy groups of spheres. We use $C \alpha$ to denote the cofiber of $\alpha$.
\end{notation}

\begin{defn}
Let $A$, $B$, $C$ and $D$ be CW spectra, $i$ and $q$ be maps
\begin{displaymath}
    \xymatrix{
A \ar@{^{(}->}[r]^-i & B, & B \ar@{->>}[r]^-q & C
    }
\end{displaymath}
We say that $(A, i)$ is an $H\mathbb{F}_2$-subcomplex of $B$, if the map $i$ induces an injection on mod 2 homology. We denote an $H\mathbb{F}_2$-subcomplex by a hooked arrow as above.

We say that $(C, q)$ is an $H\mathbb{F}_2$-quotient complex of $B$, if the map $q$ induces a surjection on mod 2 homology. We denote an $H\mathbb{F}_2$-quotient complex by a double headed arrow as above.

When the maps involved are clear in the context, we also say $A$ is an $H\mathbb{F}_2$-subcomplex of $B$, and $C$ is an $H\mathbb{F}_2$-quotient complex of $B$.

Furthermore, we say $D$ is an $H\mathbb{F}_2$-subquotient of $B$, if $D$ is an $H\mathbb{F}_2$-subcomplex of an $H\mathbb{F}_2$-quotient complex of $B$, or an $H\mathbb{F}_2$-quotient complex of an $H\mathbb{F}_2$-subcomplex of $B$.
\end{defn}

One example is that $S^1, S^3, S^7$ are $H\mathbb{F}_2$-subcomplexes of $P_1^\infty$, due to the solution of the Hopf invariant one problem. Another example is that $\Sigma^7 C \eta$ is an $H\mathbb{F}_2$-subcomplex of $P_7^9$.

We use the following way to denote the elements in the Adams $E_2$-page of $P_1^\infty$ and its $H\mathbb{F}_2$-subquotients. One way to compute $Ext(P_1^\infty)$ is to use the algebraic Atiyah-Hirzebruch spectral sequence.

\begin{displaymath}
    \xymatrix{
  E_1 = \bigoplus_{n=1}^\infty Ext(S^n) \ar@{=>}[r] & Ext(P_1^\infty)
    }
\end{displaymath}

\begin{notation}
We denote any element in $Ext(S^n)$ by $a[n]$, where $a\in Ext(S^0)$, and $n$ indicates that it comes from $Ext(S^n)$. We will abuse notation and write the same symbol $a[n]$ for an element of $Ext(P_1^\infty)$ detected by the element $a[n]$ of the Atiyah-Hirzebruch $E_\infty$-page. Thus, there is indeterminacy in the notation $a[n]$ that is detected by Atiyah-Hirzebruch $E_\infty$ elements in lower filtration. When $a[n]$ is the element of lowest Atiyah-Hirzebruch filtration in the Atiyah-Hirzebruch $E_\infty$-page in a given bidegree $(s,t)$, then $a[n]$ also is a well-defined element of $Ext(P_1^\infty)$.

We use similar notations for homotopy classes.
\end{notation}

\begin{rem}
In \cite{WX2}, we computed differentials in the algebraic Atiyah-Hirzebruch spectral sequence that converges to the Adams $E_2$-page of $P_1^\infty$ in the range of $t<72$.


By truncating the algebraic Atiyah-Hirzebruch spectral sequence for $P_1^\infty$,
one can also read off information about $Ext(P_n^{n+k})$. For details, see \cite{WX2}.
\end{rem}

\begin{rem}
Despite the indeterminacy in Notation 3.3, there is a huge advantage of it. Suppose $f: Q\rightarrow Q'$ is a map between two $H\mathbb{F}_2$-subquotients of $P_1^\infty$, and there exists an element $a[n]$ which is a generater of both $Ext^{s,t}(Q)$ and $Ext^{s,t}(Q')$ for some bidegree $(s,t)$ (this implies both $Q$ and $Q'$ have a cell in dimension $n$). We must have that, with the right choices, $a[n]$ in $Ext^{s,t}(Q)$ maps to $a[n]$ in $Ext^{s,t}(Q')$. This property follows from the naturality of the algebraic Atiyah-Hirzebruch spectral sequence.

\begin{displaymath}
    \xymatrix{
  \underset{i \in I}\bigoplus Ext(S^i) \ar@{=>}[dd] \ar[rr] & & \underset{i \in I'}\bigoplus Ext(S^i) \ar@{=>}[dd] \\
  & & \\
  Ext(Q) \ar[rr] & & Ext(Q') \\
 a[n] \ar@{|->}[rr] & & a[n]
    }
\end{displaymath}
\end{rem}

\section{the $\sigma$-extension on $h_3d_1$}

In this section, we prove part $(1)$ of Proposition 1.3. The proof can be summarized in the following ``road map" with 4 main steps:

\begin{displaymath}
    \xymatrix{
  S^0 & & & S^7 \ar@{^{(}->}[dd] \ar@{-->}[lll]_\sigma \ar@{-->}[llldddd] & & N & &  h_3d_1[7] \\
  & & & & & & & \\
  P_1^\infty \ar[uu] & & & \Sigma^7 C\eta \ar@{^{(}->}[dd] & & h_1t[9] \ar@{|->}[uu]
 & &  h_2 \cdot h_2n[9] = h_1t[9] \ar@{|->}[dd] \\
  & & & & & & & \\
  P_1^9 \ar@{^{(}->}[uu] \ar@{->>}[rrr] & & & P_7^9 & & h_1t[9] \ar@{|->}[uu] \ar@{|->}[rr] & & h_1t[9]
  }
\end{displaymath}

Here the elements in the right side of the ``road map" are elements in the 46-stem of the $E_\infty$-page of the Adams spectral sequences of the spectra in the corresponding positions.

\begin{enumerate}

\item \textbf{\underline{Step 1}}: We show that the element $h_1t[9]$ is a permanent cycle in the Adams spectral sequence of $\Sigma^7 C\eta$, and hence a permanent cycle in the Adams spectral sequence of $P_7^9$.
This is stated as Proposition 4.3.

\item \textbf{\underline{Step 2}}: Under the inclusion map $S^7 \hookrightarrow \Sigma^7 C\eta$, we show that the element $h_1t[9]$ detects the image of $\sigma\{d_1\}[7]$ in $\pi_{46}(\Sigma^7 C\eta)$. By naturality, the same statement is true, after we further map it to $\pi_{46}(P_7^9)$.
This is stated as Proposition 4.6.

\item \textbf{\underline{Step 3}}: Under the any map $S^7 \hookrightarrow P_1^9$ lifting the inclusion $S^7 \hookrightarrow P_7^9$, we show that the element $h_1t[9]$ in $Ext(P_1^9)$ detects the image of $\sigma\{d_1\}[7]$ in $\pi_{46}(P_1^9)$.
This is stated as Proposition 4.7.

\item \textbf{\underline{Step 4}}: Using the inclusion map $P_1^9 \rightarrow P_1^\infty$ and the transfer map $P_1^\infty \rightarrow S^0$, we push forward the element $h_1t[9]$ in the $E_\infty$-page of $P_1^9$ to the element $N$ in the $E_\infty$-page of $S^0$. Since the composition
\begin{displaymath}
    \xymatrix{
S^7 \ar@{^{(}->}[r] & P^9_1 \ar[r] & P_1^\infty \ar[r] & S^0
    }
\end{displaymath}
is just $\sigma$, we have the desired $\sigma$-extension from $h_3d_1$ to $N$ in the Adams spectral sequence for $S^0$.
\end{enumerate}

\begin{rem}
Step 2 is the essential step. Intuitively, it comes from the zigzag of the following two differentials:
$$d_3(e_1) = h_1t$$
in the Adams spectral sequence of $S^0$, and
$$d_2(e_1[9]) = h_1e_1[7] = h_3d_1[7]$$
in the algebraic Atiyah-Hirzebruch spectral sequence of $\Sigma^7 C\eta$ that comes from the $\eta$-attaching map. Here
$$h_1e_1=h_3d_1$$
is a relation in $Ext$. This zigzag suggested that we consider the possibility that $h_1t[9]$ detects $\sigma\{d_1\}[7]$ in $\pi_{46}(\Sigma^7 C\eta)$.
\end{rem}

We start with Step 1. Proposition 4.3 is a consequence of the following lemma.

\begin{lem}
The element	$h_2n[9]$ is a permanent cycle in $\Sigma^7 C \eta$, which detects a homotopy class that maps to $\nu\{n\}[9]$ under the quotient map
\begin{displaymath}
    \xymatrix{
\Sigma^7 C \eta \ar@{->>}[r] & S^9.
    }
\end{displaymath}
\end{lem}

\begin{proof}
The cofiber sequence
\begin{displaymath}
    \xymatrix{
 S^7 \ar@{^{(}->}[r]^i & \Sigma^7 C \eta \ar@{->>}[r]^p & S^9 \ar[r]^{\eta} & S^8
    }
\end{displaymath}
gives us a long exact sequence of homotopy groups
\begin{displaymath}
    \xymatrix{
  \pi_{43}(S^7) \ar[r]^{i_*} & \pi_{43}(\Sigma^7 C \eta) \ar[r]^{p_*} & \pi_{43}(S^9) \ar[r]^-{\eta} & \pi_{43}(S^8).
    }
\end{displaymath}

Since $h_2n$ detects $\nu\{n\}$, and
$$\eta\cdot\nu\{n\}=0,$$
there is an element $\alpha$ in $\pi_{43}(\Sigma^7 C \eta)$ such that $p_* \alpha = \nu\{n\}[9]$.

The element $h_2n[9]$ in $Ext(S^9)$ has Adams filtration 6, therefore by naturality, if it were not detected by $h_2n[9]$ in $Ext(\Sigma^7 C \eta)$, it would be detected by an element with Adams filtration at most 5.

From the same cofiber sequence, we have a short exact sequence on cohomology
\begin{displaymath}
    \xymatrix{
  0 \ar[r] & H^\ast(S^{9}) \ar[r]^{p^\ast} & H^\ast(\Sigma^{7}C\eta) \ar[r]^{i^\ast} & H^\ast(S^{7}) \ar[r] & 0
    }
\end{displaymath}
and therefore a long exact sequence of $Ext$ groups
\begin{displaymath}
    \xymatrix{
  Ext^{s-1,t-1}(S^{8}) \ar[r]^-{h_1} & Ext^{s,t}(S^{7}) \ar[r]^-{i_{\sharp}} & Ext^{s,t}(\Sigma^{7}C\eta) \ar[r]^{p_\sharp} & Ext^{s,t}(S^{9}).
    }
\end{displaymath}
This gives us the Adams $E_2$-page of $\Sigma^{7}C\eta$ in the 42 and 43 stems for $s\leq 6$ in Table 1.

\begin{table}[h]
\caption{The Adams $E_2$-page of $\Sigma^{7}C\eta$ in the 42 and 43 stems for $s\leq 6$}
\centering
\begin{tabular}{ l l l }
$s\backslash t-s$ & 42 & 43 \\ [0.5ex] 
\hline 
6 & & $h_2n[9]$\\
  & & $t[7]$\\ \hline
5 & $h_0p[9]$ & \\
  & $h_2d_1[7]$ &\\ \hline
4 & $p[9]$ & $h_0^2h_2h_5[9]$ \\ \hline
3 & & $h_0h_2h_5[9]$\\ \hline
2 & & $h_2h_5[9]$\\
\end{tabular}
\label{Ceta}
\end{table}
The element $h_2h_5[9]$ must support a nontrivial differential, since its image $p_\sharp(h_2h_5[9])$ supports a $d_3$ differential that kills $h_0p[9]$ in the Adams spectral sequence of $S^9$.

The elements $h_0h_2h_5[9]$ and $h_0^2h_2h_5[9]$ survive and detect homotopy classes that map to $\{h_0h_2h_5\}[9]$ and $\{h_0^2h_2h_5\}[9]$ in $\pi_{43}(S^9)$. In fact, since there is no $\eta$-extension on $h_0h_2h_5$ and $h_0^2h_2h_5$, we can choose homotopy classes in $\pi_{43}(S^9)$, which are detected by $h_0h_2h_5[9]$ and $h_0^2h_2h_5[9]$ and are zero after multiplying by $\eta$. Therefore, they have nontrivial pre-images under the map $p_*$ in the long exact sequence of homotopy groups. For filtration reasons, their pre-images must be detected by $h_0h_2h_5[9]$ and $h_0^2h_2h_5[9]$ in the Adams spectral sequence of $\Sigma^7 C\eta$.

Therefore, the only possibility left is $h_2n[9]$, which completes the proof.
\end{proof}

We prove Proposition 4.3 in Step 1.

\begin{prop}
The elements $h_2n[9]$ and $h_1t[9]$ are permanent cycles in the Adams spectral sequence of $\Sigma^7 C\eta$, and hence also in that of $P_7^9$. 
\end{prop}

\begin{proof}
We have a relation in $Ext$:
$$h_2 \cdot h_2n = h_1t.$$
Therefore, $h_1t[9]$ is product of permanent cycles. The second claim follows from the naturality of the Adams spectral sequences.
\end{proof}

For Step 2, we first show the following lemma.

\begin{lem}
The element $h_1t[9]$ is not a boundary in the Adams spectral sequences of $\Sigma^7 C \eta$ and $P_7^9$.
\end{lem}

\begin{proof}
The element $h_1t[9]$ is hit by a $d_3$ differential on $e_1[9]$

In the Adams spectral sequence of $S^9$, we have the Bruner differential
$$d_3(e_1[9]) = h_1t[9].$$
However, the element $e_1[9]$ is not present in either $Ext(\Sigma^7 C \eta)$ or $Ext(P_7^9)$.

Therefore, by naturality, the element $h_1t[9]$ cannot be hit by any $d_r$ differential for $r\leq 3$ in the Adams spectral sequence of $\Sigma^7 C \eta$ and $P_7^9$.

We have the Adams $E_2$-page of $\Sigma^{7}C\eta$ and $P_7^9$ in the 46 and 47 stems for $s\leq 7$ in Table 2.

\begin{table}[h]
\caption{The Adams $E_2$-page of $\Sigma^{7}C\eta$ and $P_7^9$ in the 46 and 47 stems for $s\leq 7$}
\centering
\begin{tabular}{ l | l l | l l }
& $Ext(\Sigma^{7}C\eta)$ & & $Ext(P^9_7)$ & \\ [0.5ex]
\hline
$s\backslash t-s$ & 46 & 47 & 46 & 47 \\  
\hline 
7 & $h_1t[9]$ & $\bullet$ & $h_1t[9]$ & $\bullet$\\
  & $h_0^2x[9]$ & & $h_0^2x[9]$ & $\bullet$ \\ \hline
6 & $h_0x[9]$ & $\bullet$ & $h_0x[9]$ & $\bullet$ \\
  & & $\bullet$ & $h_1x[8]$ & $\bullet$ \\ \hline
5 & & $\bullet$ & $\bullet$ & $\bullet$\\
  & & $\bullet$ & &$\bullet$\\ \hline
4 & $\bullet$ & $\bullet$ & $\bullet$ & $\bullet$\\
  & & $\bullet$ & $\bullet$ & $\bullet$\\
  & & & & $\bullet$\\ \hline
3 & $\bullet$ & $h_0h_3h_5[9]$ & $\bullet$ & $h_0h_3h_5[9]$\\ \hline
  & & & & $h_1h_3h_5[8]$\\
\end{tabular}
\label{p79}
\end{table}
We need to rule out two candidates: $h_0h_3h_5[9]$ and $h_1h_3h_5[8]$.

In the Adams spectral sequence of $S^9$, we have a $d_4$ differential:
$$d_4(h_0h_3h_5[9]) = h_0^2x[9].$$
By naturality of the quotient map to $S^9$, the element $h_0h_3h_5[9]$ cannot support a $d_4$ differential that kills $h_1t[9]$.

For the element $h_1h_3h_5[8]$, it is straightforward to check it is a permanent cycle in the Adams spectral sequence of $P_7^8$, and hence a permanent cycle in that of $P_7^9$. This rules out the candidate $h_1h_3h_5[8]$ and completes the proof.
\end{proof}

\begin{rem}
In $Ext^{6,6+46}(P_7^9)$, the element $h_1x[8]$ is clearly a surviving cycle. There are two possibilities for the other element $h_0x[9]$: it is either killed by a $d_3$ differential from $h_0h_3h_5[9]$, or it survives and detects $\{h_1h_3h_5\}[7]$. We will leave the reader to figure out which way it goes.
\end{rem}

We prove Proposition 4.6 in Step 2.

\begin{prop}
Under the inclusion map $S^7 \hookrightarrow \Sigma^7 C\eta$, the element $h_1t[9]$ detects the image of $\sigma\{d_1\}[7]$ in $\pi_{46}(\Sigma^7 C\eta)$. By naturality, the same statement is true after we further map it to $\pi_{46}(P_7^9)$.
\end{prop}

\begin{proof}
By Lemma 4.2 and Proposition 4.3, the element $h_2n[9]$ survives in the Adams spectral sequence of $\Sigma^7 C \eta$, and detects a homotopy class that maps to $\nu\{n\}[9]$ under the quotient map
\begin{displaymath}
    \xymatrix{
\Sigma^7 C \eta \ar@{->>}[r] & S^9.
    }
\end{displaymath}
By Lemma 4.4, the element $h_1t[9] = h_2 \cdot h_2n[9]$ survives and detects the homotopy class $\nu\{n\}[9] \cdot \nu$. As showed in the proof of Lemma 2.2, the element $h_3d_1 = h_1e_1$ detects an element in the Toda bracket
$$\langle \eta, \nu\{n\}, \nu\rangle.$$
Therefore, by Lemma 5.3, we have
$$\nu\{n\}[9] \cdot \nu = \langle \eta, \nu\{n\}, \nu\rangle[7] = \sigma\{d_1\}[7]$$
in $\pi_{46}(\Sigma^7 C\eta)$.
\end{proof}

Now we prove Step 3.

\begin{prop}
Under the inclusion map $S^7 \hookrightarrow P_1^9$, the element $h_1t[9]$ in $Ext(P_1^9)$ detects the image of $\sigma\{d_1\}[7]$ in $\pi_{46}(P_1^9)$.
\end{prop}

The idea of the proof of Proposition 4.7 is to make use of naturality of the Adams filtrations.

\begin{displaymath}
    \xymatrix{
S^7 \ar@{^{(}->}[r] & P_1^9 \ar@{->>}[r] & P^9_7 \\
h_3d_1[7] & & h_1t[9] \\
AF = 5 & & AF =7
    }
\end{displaymath}

The homotopy class $\sigma\{d_1\}[7]$ is detected by $h_3d_1[7]$ in $S^7$, which has Adams filtration 5, while its image in $\pi_{46}(P_7^9)$ is detected by $h_1t[9]$ by Proposition 4.6, which has Adams filtration 7. Therefore, to prove Proposition 4.7, we only need to rule out surviving cycles in the Adams filtration 6, which also lie in the kernel of the map
\begin{displaymath}
    \xymatrix{
     P_1^9 \ar@{->>}[r] & P^9_7
    }
\end{displaymath}
in the Adams $E_\infty$-page. Note that the element $h_3d_1[7]$ is not present in $Ext(P_1^9)$.

\begin{proof}
We have the Adams $E_2$-page of $P_1^9$ and $P_1^\infty$ in the 46 and 47 stems for $s\leq 8$ in Table 3.

\begin{table}[h]
\caption{The Adams $E_2$-page of $P_1^9$ and $P_1^\infty$ in the 46 and 47 stems for $s\leq 8$}
\centering
\begin{tabular}{ l | l l | l l }
& $Ext(P_1^9)$ & & $Ext(P_1^\infty)$ & \\ [0.5ex]
\hline
$s\backslash t-s$ & 46 & 47 & 46 & 47 \\  
\hline 
8 & $Ph_1^3h_5[4]$ & $\bullet$ & $Ph_1^3h_5[4]$ & $\bullet$\\
  & $\bullet$ & $\bullet$ & & \\
  & $\bullet$ & $\bullet$ & &  \\ \hline
7 & $Ph_1^2h_5[5]$ & $\bullet$ & $Ph_1^2h_5[5]$ & $\bullet$ \\
  & $h_1t[9]$ & $\bullet$ & $h_1t[9]$ & $\bullet$ \\
  & $h_0^2x[9]$ & $\bullet$ & & $\bullet$ \\
  & & $\bullet$ & & \\ \hline
6 & $Ph_1h_5[6]$ & $\bullet$ & $Ph_1h_5[6]$ & $h_1h_5d_0[1]$ \\
  & $h_0^2g_2[2]$ & $\bullet$ & $h_0^2g_2[2]$ & $h_1x[9]$ \\
  & $h_1x[8]$ & $\bullet$ & $h_1x[8]$ & \\
  & $h_0x[9]$ & & &  \\ \hline
5 & $h_0^3h_3h_5[8]$ & $\bullet$ & $h_0^3h_3h_5[8]$ & $h_1g_2[2]$ \\
  & $\bullet$ & $\bullet$ & $\bullet$ & $h_1f_1[6]$ \\
  & $\bullet$ & $\bullet$ & $\bullet$ &  \\
  & & & $\bullet$ &  \\ \hline
4 & $\bullet$ & $\bullet$ & $h_1^3h_5[12]$ & $h_0h_4^3[2]$ \\
  & & $\bullet$ & & $g_2[3]$ \\
  & & $\bullet$ & & $f_1[7]$ \\
  & & $\bullet$ & & \\ \hline
3 & $\bullet$ & $\bullet$ & $h_1^2h_5[13]$ & \\
  & $\bullet$ & & $\bullet$ & \\
  & & & $\bullet$ & \\ \hline
2 & & & $h_1h_5[14]$ & $h_2h_5[13]$\\ \hline
1 & & & $h_5[15]$ & \\
\end{tabular}
\label{p19}
\end{table}

There are 4 elements in $Ext^{6,6+46}(P_1^9)$:
$$Ph_1h_5[6], \ h_0^2g_2[2], \ h_1x[8], \ h_0x[9].$$
Remark 4.5 rules out the last two candidates, since they do not lie in the kernel of the map
\begin{displaymath}
    \xymatrix{
     P_1^9 \ar@{->>}[r] & P^9_7
    }
\end{displaymath}
in the Adams $E_\infty$-page.

In the table for the transfer map in \cite{WX2}, we have that the element $h_0^2g_2[2]$ maps to $B_1$. If the image of the homotopy class $\sigma\{d_1\}[7]$ were detected by $h_0^2g_2[2]$, then we would have a $\sigma$-extension from $h_3d_1$ to $B_1$ in $\pi_{46}S^0$, which by Lemma 2.2 is equivalent to an $\eta$-extension from $h_1g_2$ to $B_1$ in $\pi_{46}S^0$. However, the proof of Lemma 4.2.47 of \cite{Isa1} shows the latter is not true.

The only candidate left is $Ph_1h_5[6]$. To rule it out, we notice there is a long $h_0$ tower in the 46 stem of $P_1^\infty$: from $h_5[15]$ to $Ph_1^3h_5[4]$. In particular, we have
$$h_0 \cdot Ph_1h_5[6] = Ph_1^2h_5[5], \ h_0 \cdot Ph_1^2h_5[5] = Ph_1^3h_5[4].$$
Since
$$2 \cdot \sigma\{d_1\} =0,$$
the image of the homotopy class $\sigma\{d_1\}[7]$ must have order 2. Therefore, we only need to show the element $Ph_1^2h_5[5]$ is not a boundary. In the following Lemma 4.8, we show that the elements in Adams filtration 4 to 6 of $Ext(P_1^\infty)$ are all permanent cycles. This only leaves the possibility that $h_2h_5[13]$ kills $Ph_1^3h_5[4]$, but not $Ph_1^2h_5[5]$, and hence completes the proof.
\end{proof}

\begin{lem}
The elements in Adams filtration 4 to 6 of the 47-stem of $Ext(P_1^\infty)$ are all permanent cycles.
\end{lem}

\begin{proof}
There are 7 elements:
$$h_1h_5d_0[1], \ h_1x[9], \ h_1g_2[2], \ h_1f_1[6], \ h_0h_4^3[2], \ g_2[3], \ f_1[7].$$
The spheres $S^1, S^3, S^7$ are $H\mathbb{F}_2$-subcomplexes of $P_1^\infty$ by the solution of the Hopf invariant one problem. Since the elements $h_1h_5d_0, \ g_2, \ f_1$ are permanent cycles in the Adams spectral sequence for $S^0$, The elements $h_1h_5d_0[1], \ h_1g_2[2], \ f_1[7]$ are permanent cycles.

The element $h_1f_1[6] = h_0 \cdot f_1[7]$ is therefore also a permanent cycle.

It is straightforward to show that the elements $h_1g_2[2]$ and $h_0h_4^3[2]$ are permanent cycles in the Adams spectral sequence of $P_1^2$. By naturality they are permanent cycles in that of $P_1^\infty$.

For the element $h_1x[9]$, one uses the $H\mathbb{F}_2$-subcomplex of $P_1^\infty$ which contains cells in dimensions 3, 5, 7, 9 to show that it is a permanent cycle. In fact, by comparing the Atiyah-Hirzebruch spectral sequence with the Adams spectral sequence of this 4 cell complex, it follows from the following relations in the stable homotopy groups of spheres:
$$0 \in \eta \cdot \{h_1x\}, \ 0 \in \langle \nu, \eta, \{h_1x\} \rangle.$$
The homotopy class $\{h_1x\}[9]$ survives in the Atiyah-Hirzebruch spectral sequence, and is detected by $h_1x[9]$ in its Adams $E_2$-page. In particular, $h_1x[9]$ is a permanent cycle in the Adams spectral sequence of this 4 cell complex, and therefore also a permanent cycle in the Adams spectral sequence of $P_1^\infty$.
\end{proof}


Now we prove Step 4.

\begin{lem}
The element $h_1t[9]$ maps to $N$ under the transfer map.
\end{lem}

\begin{proof}
We check the two tables in the appendix of \cite{WX2}. See \cite{WX2} for more details of the Lambda algebra notation we use here.
The element $N$ is in $Ext^{8,8+46}(S^0) = \mathbb{Z}/2$. Checking the table for $P_1^\infty$, we have that
$$Ext^{7,7+46}(P_1^\infty) = (\mathbb{Z}/2)^2, \text{~~generated by~~}  (5)~ 11~ 12~ 4~ 5~ 3~ 3~ 3,\ (9)~ 3~ 5~ 7~ 3~ 5~ 7~ 7,$$
which means $Ext^{7,7+46}(P_1^\infty)$ is generated by $Ph_1^2h_5[5]$ and $h_1t[9]$. Since $Ph_1^2h_5[5]$ is divisible by $h_0$ in $Ext(P_1^\infty)$, while $N$ is not divisible by $h_0$ in $Ext(S^0)$, $Ph_1^2h_5[5]$ cannot map to $N$ under the transfer map. By the algebraic Kahn-Priddy theorem \cite{Lin}, the other generator $h_1t[9]$ has to map to $N$.
\end{proof}

\section{A lemma for extensions in the Atiyah-Hirzebruch spectral sequence}

Let $\alpha:Y\rightarrow X$ and $\beta:Z\rightarrow Y$ be homotopy classes of maps between spectra. Suppose that the composite $\alpha\beta=0$.
Let $C\alpha$ and $C\beta$ be the cofiber of $\alpha$ and $\beta$ respectively. 

We have cofiber sequences:
$$Y\xrightarrow{\alpha}X\xrightarrow{i_\alpha}C\alpha\xrightarrow{\partial_\alpha}\Sigma Y,$$
$$Z\xrightarrow{\beta}Y\xrightarrow{i_\beta}C\beta\xrightarrow{\partial_\beta}\Sigma Z.$$

Denote by $L^\alpha\beta$ be the set of maps in $[\Sigma Z, C\alpha]$ such that the composite 
$$\Sigma Z\rightarrow C\alpha \xrightarrow{\partial_\alpha}\Sigma  Y$$ is $-\Sigma\beta$.  The indeterminacy of the set $L^\alpha\beta$ is $$i_\alpha\cdot[\Sigma Z, X].$$

Similarly, denote by $L_\beta\alpha$ be the set of maps in $[C\beta, X]$ such that the composite 
$$Y\xrightarrow{i_\beta} C\beta \rightarrow X$$ is $\alpha$. The indeterminacy of the set $L_\beta\alpha$ is $$[\Sigma Z, X]\cdot \partial_\beta.$$

\begin{lem}
The two sets of maps
$L^\alpha\beta\cdot\partial_\beta$ and $i_\alpha\cdot L_\beta\alpha$ in $[C\beta, C\alpha]$ are equal.
\end{lem}

\begin{proof}

It is clear that the indeterminacy of the two sets are given by the following composition 
$$i_\alpha\cdot[\Sigma Z,X]\cdot\partial_\beta.$$
We need to show that they contain one common element. We have the following diagram
	\begin{displaymath}
    \xymatrix{
   & Y \ar[r]^{\alpha} \ar@{=}[dd] & X \ar[r]^{i_\alpha} & C\alpha \ar[r]^{\partial_\alpha} & \Sigma Y \ar[r]^{-\Sigma\alpha} & \Sigma X \\
     & & & & \\
     Z \ar[r]^{\beta} & Y \ar[r]^{i_\beta} & C\beta \ar@{-->}[uu]^{f} \ar[r]^{\partial_\beta} & \Sigma Z \ar@{-->}[uu]^{g} \ar[r]^{-\Sigma\beta} & \Sigma Y \ar@{=}[uu] &
      }
\end{displaymath}

Take $f\in L_\beta\alpha$. Since both lines are cofiber sequences, there exists a co-extension $g\in L^\alpha\beta$ such that the diagram commutes. The commutativity of the middle square gives the claim.
\end{proof}

\begin{lem}
Let $$W\xrightarrow{\gamma}Z\xrightarrow{\beta}Y\xrightarrow{\alpha}X$$ be a sequence of homotopy classes of maps.
Suppose that $\alpha\beta=0$ and $\beta\gamma=0$.
Then the two sets of maps
$-L^\alpha\beta\cdot\Sigma\gamma$ and $ i_\alpha\cdot L_\beta\alpha\cdot L^\beta\gamma$ in $[\Sigma W, C\alpha]$ are equal.
\end{lem}
\begin{proof}
First, the indeterminacy of the former is $i_\alpha\cdot[\Sigma Z,X]\cdot\Sigma\gamma$. The indeterminacy of the latter is $$i_\alpha\cdot[\Sigma Z,X]\cdot \partial_\beta\cdot L^\beta\gamma + i_\alpha\cdot L_\beta\alpha\cdot i_\beta\cdot[\Sigma W, Y]$$
Note that $ \partial_\beta\cdot L^\beta\gamma = - \Sigma\gamma$ and $i_\alpha\cdot L_\beta\alpha\cdot i_\beta = i_\alpha\cdot \alpha =0$. So the two sets have the same indeterminacy.

We have the following diagram:
\begin{displaymath}
\xymatrix{
W \ar[d]^{\Sigma^{-1}L^\beta\gamma} \ar[rr]^\gamma && 
Z \ar@{=}[d] \ar[rr] ^{i_\gamma} && 
C\gamma \ar[d]^{L_\gamma\beta} \ar[r]^{\partial_\gamma} &\Sigma W \ar[d]^{L^\beta\gamma}\ar[r]^{-\Sigma\gamma}& \Sigma Z \ar@{=}[d] \\
\Sigma^{-1}C\beta \ar[d]^{\Sigma^{-1}L_\beta\alpha} \ar[rr]^{-\Sigma^{-1}\partial_\beta} && 
Z \ar[d]^{\Sigma^{-1}L^\alpha\beta} \ar[rr]^\beta && Y \ar@{=}[d] \ar[r]^{i_\beta} & C\beta \ar[d]^{L_\beta\alpha} \ar[r] ^{\partial_\beta} & \Sigma Z \ar[d]^{L^\alpha\beta} \ar[r]^{-\Sigma \beta} &\Sigma Y \ar@{=}[d] \\
\Sigma^{-1}X\ar[rr]^{-\Sigma^{-1}i_\alpha}&&
\Sigma^{-1}C\alpha\ar[rr]^{-\Sigma^{-1}\partial_\alpha}&& Y\ar[r]^\alpha & X\ar[r]^{i_\alpha} &C\alpha\ar[r]^{\partial_\alpha}&\Sigma Y
}
\end{displaymath}
By Lemma 5.1, with suitable choices, all the squares commute, then claim follows. In fact, taking any choices of $L^\beta\gamma$ and $L_\beta\alpha$, Lemma 5.1 says there exist choices for $L_\gamma\beta$ and $L^\alpha\beta$, making the diagrams commute.
\end{proof}

Now we have the following lemma as a corollary of Lemma 5.2 when the spectra $X, \ Y, \ Z, \ W$ are all spheres.

\begin{lem}
Let $\alpha, \ \beta$ and $\gamma$ be maps between spheres.
$$\alpha:S^{|\alpha|}\rightarrow S^0, \  
	\beta:S^{|\alpha|+|\beta|}\rightarrow S^{|\alpha|}, \ \gamma:S^{|\alpha|+|\beta|+|\gamma|}\rightarrow S^{|\alpha|+|\beta|}.$$ 
	
	Then in the Atiyah-Hirzebruch spectral sequence of $C\alpha$, we have a $\gamma$-extension
	$$ \beta[|\alpha|+1]\cdot\Sigma\gamma = \langle \alpha, \beta, \gamma \rangle[0]$$ 
\end{lem}

\begin{proof}
	By definition, the set of classes represented by $\beta[|\alpha|+1]$ in the Atiyah-Hirzebruch spectral sequence is $-L^\alpha\beta$. On the other hand, by definition, $L_\beta\alpha\cdot L^\beta\gamma$ is $\langle \alpha,\beta,\gamma \rangle$, and $i_\alpha\cdot L_\beta\alpha\cdot L^\beta\gamma$ is $\langle \alpha,\beta,\gamma \rangle[0]$. So the claim follows from Lemma 5.2.
\end{proof}

\section{Appendix}

In this appendix, we use cell diagrams as intuition for the statements of the lemmas in Section 5.  It is very helpful when thinking of CW spectra. See \cite{BJM, WX1, Xu} for example. For simplicity, we restrict to the cases when the spectra $X, \ Y, \ Z, \ W$ are all spheres. For the definition of cell diagrams, see \cite{BJM}.

Let $\alpha, \ \beta$ be classes in the stable homotopy groups of spheres such that $\beta \cdot \alpha =0$. We denote the cofiber of $\alpha$ by
\begin{displaymath}
    \xymatrix{
 *+[o][F-]{} \ar@{-}[d]^{\alpha} \\
 *+[o][F-]{}}
\end{displaymath}
We denote the maps $i_\alpha$ and $\partial_\alpha$ by
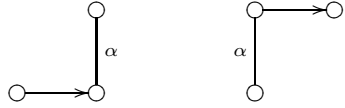
\begin{displaymath}
    \xymatrix{
 & *+[o][F-]{} \ar@{-}[d]^{\alpha} & & *+[o][F-]{} \ar@{-}[d]_{\alpha} \ar[r] & *+[o][F-]{} \\
 *+[o][F-]{} \ar[r] & *+[o][F-]{} & & *+[o][F-]{} & }
\end{displaymath}
and the extension and co-extension maps $L_\alpha \beta$ and $L^\alpha \beta$ by
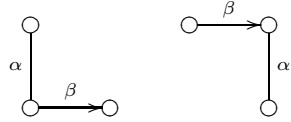
\begin{displaymath}
    \xymatrix{
*+[o][F-]{} \ar@{-}[d]_{\alpha} & & *+[o][F-]{} \ar[r]^{\beta} & *+[o][F-]{} \ar@{-}[d]^{\alpha} \\
 *+[o][F-]{} \ar[r]^{\beta} & *+[o][F-]{} & & *+[o][F-]{} }
\end{displaymath}


Then Lemma 5.1 says the following two sets of maps are equal:
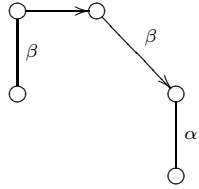
\begin{displaymath}
    \xymatrix{
 *+[o][F-]{} \ar@{-}[d]^{\beta} \ar[r] & *+[o][F-]{} \ar[dr]^{\beta} & \\
 *+[o][F-]{} & & *+[o][F-]{} \ar@{-}[d]^{\alpha} \\
  & & *+[o][F-]{}}
\end{displaymath}

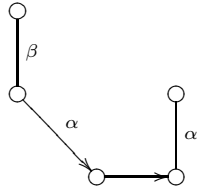
\begin{displaymath}
    \xymatrix{
     *+[o][F-]{} \ar@{-}[d]^{\beta} & & \\
    *+[o][F-]{} \ar[dr]^{\alpha} & & *+[o][F-]{} \ar@{-}[d]^{\alpha} \\
   & *+[o][F-]{} \ar[r] & *+[o][F-]{} }
\end{displaymath}


with the same indeterminacy:
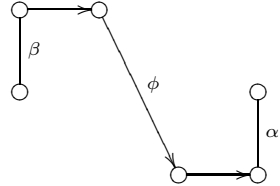
\begin{displaymath}
    \xymatrix{
     *+[o][F-]{} \ar@{-}[d]^{\beta} \ar[r] & *+[o][F-]{} \ar[ddr]^{\phi} & & \\
    *+[o][F-]{} & & & *+[o][F-]{} \ar@{-}[d]^{\alpha} \\
   & & *+[o][F-]{} \ar[r] & *+[o][F-]{} }
\end{displaymath}
where $\phi\in\pi_{|\alpha|+|\beta|+1}S^0$ could be any class. 


Suppose further that $\beta\cdot\gamma=0$. Pre-composing with $L^\beta \gamma$, Lemma 5.3 says that the following two sets of maps are equal:
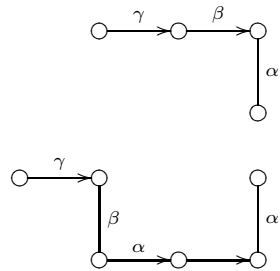
\begin{displaymath}
    \xymatrix{
  & *+[o][F-]{} \ar[r]^{\gamma} & *+[o][F-]{} \ar[r]^{\beta} & *+[o][F-]{} \ar@{-}[d]^{\alpha} \\
  & & & *+[o][F-]{}}
\end{displaymath}

\begin{displaymath}
    \xymatrix{
     *+[o][F-]{} \ar[r]^{\gamma} & *+[o][F-]{} \ar@{-}[d]^{\beta} & & *+[o][F-]{} \ar@{-}[d]^{\alpha} \\
   & *+[o][F-]{} \ar[r]^{\alpha} & *+[o][F-]{} \ar[r] & *+[o][F-]{} }
\end{displaymath}



\begin{thebibliography}{99}

\bibitem{BJM}
M.G. Barratt, J.D.S. Jones and M.E. Mahowald.
Relations amongst Toda brackets and the Kervaire invariant in dimension $62$.
J. London Math. Soc. 30(1984), 533--550.

\bibitem{BMT}
M.G. Barratt, M.E. Mahowald and M.C.Tangora.
Some differentials in the Adams spectral sequence. II
Topology. 9(1970), 309--316.

\bibitem{Br1}
Robert Bruner.
A new differential in the Adams spectral sequence.
Topology 23(1984), 271-276.

\bibitem{Br2}
Robert Bruner.
The cohomology of the mod 2 Steenrod algebra: a computer calculation. http://www.math.wayne.edu/~rrb/papers/cohom.pdf

\bibitem{Isa1}
Daniel C. Isaksen.
Stable stems.
arXiv:1407.8418.

\bibitem{Isa2}
Daniel C. Isaksen.
Classical and motivic Adams charts.
arXiv:1401.4983.

\bibitem{IX}
Daniel C. Isaksen and Zhouli Xu.
Motivic stable homotopy and the stable 51 and 52 stems.
Topology and its Applications. Volume 190(2015), 31--34.

\bibitem{Koc}
Stanley O. Kochman.
Stable homotopy groups of spheres.
Lecture Notes in Mathematics, vol. 1423, Springer-Verlag, Berlin, 1990. A computer-assisted approach.

\bibitem{KM}
Stanley O. Kochman and Mark E. Mahowald.
On the computation of stable stems.
Contemporary Mathematics 181 (1993)299-316.

\bibitem{KP}
Daniel S. Kahn and Stewart B. Priddy.
The transfer and stable homotopy theory.
Math. Proc. Cambridge Philos. Soc., 83(1):103--111, 1978.

\bibitem{Lin}
Wen Hsiung Lin.
Algebraic Kahn-Priddy theorem.
Pacific J. Math., 96 (1981), 435--455.

\bibitem{May}
J. Peter May.
Matric Massey products.
J. Algebra 12(1969), 533--568.

\bibitem{Mos}
R. M. F. Moss.
Secondary compositions and the Adams spectral sequence.
Math. Z. 115(1970), 283--310.

\bibitem{MT}
Mark Mahowald and Martin Tangora.
Some differentials in the Adams spectral sequence.
Topology 6 (1967) 349--369.

\bibitem{Rav}
Douglas C. Ravenel.
Complex cobordism and stable homotopy groups of spheres.
Pure and Applied Mathematics, vol. 121, Academic Press, Inc., Orlando, FL, 1986.

\bibitem{Tan1}
Martin C. Tangora.
Some Massey products in Ext.
Topology and representation theory (Evanston, IL, 1992), 269-280, Contemp. Math., 158, Amer. Math. Soc., Providence, RI, 1994.

\bibitem{Tan2}
Martin C. Tangora.
On the cohomology of the Steenrod algebra.
Math. Z. 116(1970), 18-64.

\bibitem{Tan3}
Martin C. Tangora.
Some extension problems in the Adams spectral sequence.
Aarhus Univ., Aarhus, 1970. Mat. Inst.,Aarhus Univ., Aarhus, 1970, pp. 578-587. Various Publ. Ser., No. 13.

\bibitem{Tod}
Hirosi Toda.
Composition methods in homotopy groups of spheres.
Annals of Mathematics Studies 49, Princeton University Press, ISBN 978-0-691-09586-8.

\bibitem{WX1}
Guozhen Wang and Zhouli Xu.
The triviality of the 61-stem in the stable homotopy groups of spheres.
Annals of Mathematics. Vol.186 (2017), Issue 2, 501-580.

\bibitem{WX2}
Guozhen Wang and Zhouli Xu.
The algebraic Atiyah-Hurzebruch spectral sequence of real projective spectra.
arXiv:1601.02185.

\bibitem{Xu}
Zhouli Xu.
The Strong Kervaire invariant problem in dimension 62.
Geometry and Topology 20-3 (2016), 1611--1624.

\end{thebibliography}
\end{document}